\documentclass{article}

\usepackage{amsmath}
\usepackage{amsthm}
\usepackage{amssymb}

\newtheorem{theorem}{Theorem}
\newtheorem*{maintheorem}{Theorem A}
\newtheorem*{corollaryb}{Corollary B}
\newtheorem*{corollaryc}{Corollary C}
\newtheorem{prop}{Proposition}
\newtheorem{lemma}{Lemma}
\newtheorem{claim}{Claim}
\newtheorem{corollary}{Corollary}
\theoremstyle{definition}
\newtheorem{definition}{Definition}
\newtheorem{remark}{Remark}

\newcommand{\R}{\mathbb{R}}
\newcommand{\Z}{\mathbb{Z}}
\newcommand{\C}{\mathbb{C}}
\newcommand{\tX}{\widetilde{X}}
\newcommand{\ded}[1]{(\!(#1)\!)}
\newcommand{\bigded}[1]{\left(\!\!\left(#1\right)\!\!\right)}

\DeclareMathOperator{\Id}{Id}
\DeclareMathOperator{\sign}{sign}

\DeclareMathOperator{\Aut}{Aut}
\DeclareMathOperator{\Span}{Span}
\DeclareMathOperator{\diag}{diag}
\DeclareMathOperator{\Rm}{Rm}

\begin{document}

\title{Quotients of gravitational instantons}
\author{Evan P. Wright\footnotemark \\ \small Mathematics Department \\ \small Stony Brook University, Stony Brook NY, 11794}
\footnotetext[1]{evanpw@math.sunysb.edu}
\date{}
\maketitle

\begin{abstract}
A classification result for Ricci-flat anti-self-dual asymptotically locally Euclidean 4-manifolds is obtained: they are either hyperk\"ahler (one of the gravitational instantons classified by Kronheimer), or they are a cyclic quotient of a Gibbons-Hawking space. The possible quotients are described in terms of the monopole set in $\mathbb{R}^3$, and it is proved that every such quotient is actually K\"ahler. The fact that the Gibbons-Hawking spaces are the only gravitational instantons to admit isometric quotients is proved by examining the possible fundamental groups at infinity: most can be ruled out by the classification of 3-dimensional spherical space form groups, and the rest are excluded by a computation of the Rohklin invariant (in one case) or the eta invariant (in the remaining family of cases) of the corresponding space forms.
\end{abstract}

\section{Introduction}

In proving a compactness theorem for some family of Riemannian manifolds, the method of proof is often the same: show that the only possible non-compactness occurs through some sort of ``bubbling'', and then using the geometry of the situation, rule out every possible bubble. When the family consists of anti-self-dual 4-manifolds with bounded Ricci curvature (K\"ahler-Einstein surfaces, for example), the possible bubbles are Ricci-flat, anti-self-dual, and asymptotically locally Euclidean (ALE).

The simply-connected such examples were classified by Kronheimer \cite{Kronheimer1}, \cite{Kronheimer2}, but even when examining a family of metrics on a fixed, simply-connected manifold, it is difficult to rule out non-simply-connected bubbles a priori. Moreover, if one wants to rule out bubbling by arguing that not enough curvature accumulates to form a bubble, then the complete list of non-simply-connected examples is needed, since quotients will have less total curvature than their covers.

To this end, we prove the following classification result:

\begin{maintheorem}
Every Ricci-flat anti-self-dual ALE 4-manifold $X$ which is not simply-connected and not flat is a finite isometric quotient of a Gibbons-Hawking space $\tX$, and is actually K\"ahler. Moreover, if the monopole set $F \subset \R^3$ of $\tX$ is normalized to have Euclidean center of mass at the origin, then the isometric quotients of $\tX$ are in one-to-one correspondence with the cyclic subgroups of $SO(3)$ which preserve $F$ and act freely on it.
\end{maintheorem}

\begin{remark}
The fact that every Ricci-flat \emph{K\"ahler} ALE 4-manifold is either hyperk\"ahler (one of Kronheimer's examples) or a cyclic quotient of a Gibbons-Hawking space was stated without proof by Nakajima \cite{Nakajima}, and a proof of that same fact has recently been given by Ioana Suvaina. The difference with our theorem is that we do not need to assume a priori that the manifold is K\"ahler; it turns out during the proof that every isometric quotient preserves some parallel complex structure. In fact, Nakajima conjectures that the only simply-connected Ricci-flat ALE 4-manifolds are Kronheimer's hyperk\"ahler manifolds, which together with Theorem~A would imply that \emph{every} Ricci-flat ALE 4-manifold was K\"ahler.
\end{remark}

From the main theorem, the following corollaries follow quickly (they will be proved in Section~\ref{S:Consequences}):

\begin{corollaryb}
If $X$ is a Ricci-flat anti-self-dual ALE 4-manifold which is not flat, then
\[
\int_X |\Rm|^2 d\mu \geq
\begin{cases}
6\pi^2 & \text{if } b_2(X) = 0\\
8\pi^2 \left(b_2(X) + 1 - \frac{1}{b_2(X) + 1} \right) & \text{if } b_2(X) > 0.
\end{cases}
\]
In particular, $\int_X |\Rm|^2 d\mu \geq 6\pi^2$ in all cases.
\end{corollaryb}

\begin{corollaryc}
If $X$ is a scalar-flat K\"ahler ALE 4-manifold (of any order $\tau > 0$) and $b_2(X) = 0$, then $X$ is actually Ricci-flat (see the comments before Lemma~18 in \cite{CLW} for another proof of this fact), and is a quotient of a Gibbons-Hawking space $\tX$ by a cyclic group of order $\chi(\tX)$.
\end{corollaryc}

\section{ALE Spaces and Gravitational Instantons}

We say that a Riemannian manifold $(M^4, g)$ is ALE of order $\tau > 0$ if
\begin{itemize}
\item[(i)]{there is a compact set $C \subset M$ such that $M \setminus C$ is diffeomorphic to $(\R^4 \setminus B) / \Gamma$, where $B$ is some closed ball centered at the origin of $\R^4$, and $\Gamma$ is a finite subgroup of $SO(4)$ (well-defined up to conjugacy) which acts freely on the 3-sphere; and}
\item[(ii)]{$g$ approximates the Euclidean metric in the sense that the pullback of the metric on $M \setminus C$ to the cover $\R^4 \setminus B$ satisfies
\[
\partial^{\alpha} (g_{ij} - \delta_{ij}) = O(r^{-\tau-|\alpha|}) \quad \mbox{for all multi-indices } \alpha,
\]
where $|\alpha|$ denotes the length of the multi-index, and $r$ is the the function given by the distance to 0 in the Euclidean metric.}
\end{itemize}
The group $\Gamma$ is called the fundamental group at infinity of $M$, and is denoted by $\pi_1^\infty(M)$. This is the fundamental group of the space form $S^3 / \Gamma$ which lies at the ``boundary'' of $M$.

\begin{definition} \label{D:Bubble}
For brevity, we will call a Ricci-flat anti-self-dual ALE 4-manifold a \emph{bubble}.
\end{definition}

\begin{remark}
A theorem of Bando, Kasue, and Nakajima \cite{BKN} implies that every bubble is ALE of order 4 in suitably-chosen coordinates. Therefore, we will assume that coordinates at infinity are always of this order.
\end{remark}

On a real 4-manifold, there is locally no distinction between Ricci-flat anti-self-dual metrics, Ricci-flat K\"ahler metrics, and hyperk\"ahler metrics: all three are locally characterized by having their restricted holonomy in $SU(2) \cong Sp(1)$. However, globally we have strict containment:
\begin{equation*}
\{\mbox{hyperk\"ahler}\} \subset \{\mbox{Ricci-flat K\"ahler}\} \subset \{\mbox{Ricci-flat anti-self-dual}\}.
\end{equation*}
In particular, a Ricci-flat K\"ahler manifold must additionally have global holonomy in $U(2)$, and a hyperk\"ahler manifold must have global holonomy in $SU(2)$. For example, a K\"ahler-Einstein K3 surface is hyperk\"ahler, a K\"ahler-Einstein Enriques surface is only Ricci-flat K\"ahler, and an isometric quotient of a K\"ahler-Einstein Enriques surface by an antiholomorphic involution is only Ricci-flat anti-self-dual.

Those bubbles which are hyperk\"ahler (which includes all simply-connected bubbles) are also known as \emph{gravitational instantons}, and they have all been classified by Kronheimer \cite{Kronheimer1}, \cite{Kronheimer2}. In particular, every hyperk\"ahler bubble $M$ is diffeomorphic to the minimal resolution of $\C^2 / \Gamma$ for some finite subgroup $\Gamma < SU(2)$, and every such $\Gamma$ occurs. The preimage of the singular point $0$ under the resolution map $M \to \C^2 / \Gamma$ is a collection of $(-2)$-curves which intersect transversely in a pattern given by a certain Dynkin diagram associated to $\Gamma$ (see \cite{McKay}, \cite{Kronheimer1} for details). It follows that $H_2(M, \Z)$ is freely generated by the $(-2)$-curves, and so the intersection form is negative definite. Moreover, every gravitational instanton actually turns out to be simply connected. Some of the invariants of these manifolds that we will need are collected in Table~\ref{Table:Kronheimer}.

\begin{table}[ht] \caption{Invariants of gravitational instantons}
\centering
\begin{tabular}{|c|c|c|c|c|c|}
\hline
Dynkin diagram & $\Gamma$ & $|\Gamma|$ & $\chi$ & $\tau$ \\
\hline
\hline
$A_{k - 1}$ & $\Z_k$ & $k$ & $k$ & $-k + 1$ \\
$D_{k + 2}$ & $D_{4k}^*$ & $4k$ & $k + 3$ & $-k - 2$\\
$E_6$ & $T^*$ & $24$ & $7$ & $-6$ \\
$E_7$ & $O^*$ & $48$ & $8$ & $-7$ \\
$E_8$ & $I^*$ & $120$ & $9$ & $-8$ \\
\hline
\end{tabular} \label{Table:Kronheimer}
\end{table}

\subsection{Subgroups of $SU(2)$}

Because $SU(2)$ is isomorphic to $Spin(3) \cong S^3$, there is a spin double covering $\phi: SU(2) \to SO(3)$. This homomorphism is just the quotient by the center $Z(SU(2)) = \{\pm \Id\}$. From this, we can read off the finite subgroups of $SU(2)$. The subgroups which don't contain $-\Id$ must have odd order, and turn out to all be cyclic. Every other subgroup of $SU(2)$ maps $2$-to-$1$ onto a subgroup in $SO(3)$, which are all known: they are cyclic, dihedral ($D_{2n}$), tetrahedral ($T \cong A_4$), octahedral ($O \cong S_4$), or icosahedral ($I \cong A_5$). The preimage of every cyclic group is cyclic, and the preimages of the latter four cases are called binary dihedral ($D_{4n}^*$), binary tetrahedral ($T^*$), binary octahedral ($O^*$), and binary icosahedral ($T^*$).

\section{Setup and strategy}

Let $X$ be a bubble which is not simply-connected and not flat. By the splitting theorem, it follows that the fundamental group of $X$ is finite, and that $X$ has one end. By the characterization of Bando, Kasue, and Nakajima \cite{BKN}, the universal cover $\tX$ is also ALE, but since it is simply-connected, it must be hyperk\"ahler and thus is one of the gravitational instantons constructed by Kronheimer. Let $p: \tX \to X$ be the corresponding covering map.

\begin{remark} \label{R:Cohomology}
From the corresponding properties for $\tX$ and the injectivity of $p^*$ on real cohomology, we know that $b_1(X) = b_3(X) = b_4(X) = 0$, so $\chi(X) = 1 + b_2(X)$. Moreover, since $\tX$ has negative-definite intersection form, so does $X$, and thus $\tau(X) = -b_2(X) = 1 - \chi(X)$.
\end{remark}

By the theorems of Kronheimer \cite[theorems 1.2 and 1.3]{Kronheimer2} (see also \cite{Polygons}), the universal cover $\tX$ is Gibbons-Hawking if and only if the fundamental group at infinity is cyclic. Thus, in order to prove the first part of Theorem~A, we seek to rule out all other possibilities for $\pi_1^\infty(\tX)$. The main tool is the following:

\begin{prop} \label{P:ExactSequence}
There are maps $\pi^\infty_1(\tX) \to \pi^\infty_1(X)$ and $\pi^\infty_1(X) \to \pi_1(X)$ so that the sequence
\begin{equation} \label{E:ExactSequence}
1 \to \pi^\infty_1(\tX) \to \pi^\infty_1(X) \to \pi_1(X) \to 1
\end{equation}
is exact.
\end{prop}
\begin{proof}
Let $C$ be a compact subset of $X$ such that $X \setminus C$ is diffeomorphic to $(\R^4 \setminus B) / \Gamma_X$ for some subgroup $\Gamma_X \leq SO(4)$, where $B$ is a closed ball in $\R^4$. Now, the restriction of the covering map $p: \tX \to X$ to the set $p^{-1}(X \setminus C)$ is a normal covering of $X \setminus C$ with the same deck group, and we wish to show that it is connected.

Since $X \setminus C$ is a path-connected manifold, every component of a covering map must surject. (This is because every component is some quotient of a disjoint union of copies of the universal cover \cite[p. 69]{Hatcher}). Thus, no component of $p^{-1}(X \setminus C)$ can have compact closure in $\tX$. But we know that $\tX$ has only one end by the splitting theorem, and $p^{-1}(X \setminus C) = \tX \setminus p^{-1}(C)$ is the complement of a compact set, since finite coverings are proper. Thus, there can only be one component, since each one would correspond to a different end.

Now, since $\tX \setminus p^{-1}(C)$ is a connected covering space of $X \setminus C \cong (\R^4 \setminus B) / \Gamma_X$, it must also be of the form $(\R^4 \setminus B) / \Gamma$, and this $\Gamma$ must be the fundamental group at infinity of $\tX$. Then the map of fundamental groups induced by the restriction of $p$ identifies $\pi_1^\infty(\tX)$ with a normal subgroup of $\pi_1(X \setminus C) \cong \pi_1^\infty(X)$. Moreover, the quotient of $\pi_1(X \setminus C)$ by $\pi_1^\infty(\tX)$ is identified with the deck group of $p$, which is just $\pi_1(X)$.
\end{proof}

By Kronheimer's work, the possibilities for $\pi^\infty_1(\tX)$ are exactly the finite subgroups of $SU(2)$. On the other hand, the fundamental group at infinity of $X$ could a priori be any 3-dimensional spherical space form group.

\subsection{Space form groups}

The task of classifying all 3-dimensional spherical space forms is equivalent to i) finding all finite groups which admit a faithful representation into $SO(4)$ which acts freely on the 3-sphere, and then ii) classifying all such representations up to equivalence and group automorphisms. A nice exposition in arbitrary dimension can be found in the book by Wolf \cite{Wolf}. In particular, the solution to problem i) in dimension 3 is the following:

\begin{prop} \label{P:SpaceForms}
Every subgroup of $SO(4)$ which acts freely on the 3-sphere is isomorphic to
\begin{enumerate}
\item{the trivial group $1$,}
\item{a binary dihedral group $D^*_{4p}$ for some $p > 1$,}
\item{the binary tetrahedral group $T^*$,}
\item{the binary octahedral group $O^*$,}
\item{the binary icosahedral group $I^*$,}
\item{one of the groups $T'_{8 \cdot 3^k}$ for $k > 1$,}
\item{one of the groups $D'_{2^{k + 2} p}$ for $k$ greater than $0$ and $p$ an odd integer larger than $1$,}
\end{enumerate}
or a product of one of these groups with a cyclic group of relatively prime order (so in particular, all cyclic groups are space form groups).
\end{prop}

Presentations for individual groups will be given when we need them.

\begin{remark} \label{R:Conventions}
To minimize ambiguity, subscripts in the names of groups will always denote the order of the group: e.g., the binary dihedral group of order 8 is $D_{8}^*$, not $D_{2}^*$. Notice also the convention on the indices in cases 2, 6, and 7: this is to eliminate incidental isomorphisms between groups on the list, and so that we can say things like ``no binary dihedral group is abelian'', even though $D_{4}^* \cong \Z_4$.
\end{remark}

\begin{remark} \label{R:SpaceForms}
It follows from this proposition that every space form group is either a cyclic group of odd order or it has a unique element of order 2, and that every Sylow 2-subgroup of a space form group is either cyclic or binary dihedral of the form $D_{4 \cdot 2^k}^*$.
\end{remark}

Using the same trick (and the same homomorphism) we used to identify the finite subgroups of $SU(2)$, it is easy to identify the normal subgroups of the binary polyhedral groups:
\begin{prop} \label{P:normal}
Every proper normal subgroup of $D^*_{4b}$ is either cyclic, or if $b$ is even, a binary dihedral group of index 2. The non-trivial proper normal subgroups of $T^*$ are isomorphic to $\Z_2$ or $D_8^*$; those of $O^*$ are isomorphic to $\Z_2$, $D_8^*$, or $T^*$; and the only non-trivial proper normal subgroup of $I^*$ is $\Z_2$.
\end{prop}

We will also be interested in certain supergroups of the binary polyhedral groups:
\begin{prop} \label{P:split}
If a space form group of the form $\Z_m \times G$ contains a non-cyclic binary polyhedral group $B$ (i.e.,  $D_{4b}^*$, $T^*$, $O^*$, or $I^*$) (resp. as a normal subgroup), then $G$ also contains $B$ (resp. as a normal subgroup).
\end{prop}
\begin{proof}
Let $\phi: B \hookrightarrow \Z_m \times G$ be an injective homomorphism, and denote by $\phi_1$ and $\phi_2$ the component homomorphisms. If $a \in B$ is in the kernel of $\phi_2$, then $\phi(a)$ lies in the subgroup $\Z_m \times \{e\}$, which is in the center of $\Z_m \times G$. Since $\phi$ is injective, this implies that $a$ must be in the center of $B$, which is exactly the center $\{\pm \Id\}$ of $SU(2)$. This is easy to see if $B = T^*$, $O^*$, or $I^*$, since the quotients $A_4$, $S_4$, and $A_5$ by the center of $SU(2)$ are centerless. For $D_{4b}^*$, the presentation $D_{4b}^* = \langle \, a, x \mid a^{2b} = e, x^2 = a^b, x^{-1} a x = a^{-1} \,\rangle$ shows that the center has 2 elements, and $\{\pm \Id\}$ is the unique 2-element subgroup of $SU(2)$.

Thus, if $\phi_2$ is not an injection, then its kernel is exactly the center of $B$. But then $\phi_2$ descends to an injection of the polyhedral group $B / \{\pm 1\}$ into the space form group $G$. This is impossible, because none of the polyhedral groups are 3-dimensional space form groups. In particular, space form groups contain at most one element of order 2, but $D_{2b}$ contains $b$ reflection, $T \cong A_4$ contains both $(1 2)(3 4)$ and $(1 3)(2 4)$, and $O$ and $I$ both contain $T$. Thus, we have a contradiction and so $\phi_2$ gives an injection of $B$ into $G$.

Finally, notice that projection onto a direct product factor takes normal subgroups to normal subgroups.
\end{proof}

We are now prepared to look at each possibility for $\pi_1^\infty(\tX)$ to see which ones can actually occur.

\section{Exceptional cases}

\subsection{Binary icosahedral}

Suppose that the fundamental group at infinity of $\tX$ is the binary icosahedral group $I^*$. From Table~\ref{Table:Kronheimer}, we know that $\chi(\tX) = 9$, so the degree of the covering map $p: \tX \to X$ is either 3 or 9. In the latter case, since every group with 9 elements has a normal subgroup with 3 elements, there is a covering space of $X$ with fundamental group of order 3, and so in order to rule out this case, it suffices to assume that $\pi_1(X) = \Z_3$. This means that the exact sequence (\ref{E:ExactSequence}) takes the form
\[
1 \to I^* \to G \to \Z_3 \to 1
\]
for some spherical space form group $G$.

To determine all such extensions $G$, we need to know the center of $I^*$ (isomorphic to $\Z_2$, as we saw in the proof of Proposition~\ref{P:split}), and also the outer automorphism group of $I^*$ (also isomorphic to $\Z_2$ \cite[p.195]{Wolf}). Since $|\Z_3| = 3$ is relatively prime to $|Z(I^*)| = 2$, every such extension splits \cite[11.3.8 \& 11.4.10]{Robinson}, and since $|\Z_3|$ is relatively prime to $|\text{Out}(I^*)|$, the only extension is the \emph{direct} product $I^* \times \Z_3$. But this is not a space form group, because $3$ is not relatively prime to the order of $I^*$. Therefore, this case is impossible.

\subsection{Binary octahedral}

Now suppose that $\pi_1^\infty(\tX)$ is binary octahedral. In this case, $\chi(\tX) = 8$, so the degree of the covering must divide 8. However, there are only a few groups with order dividing 8, and one can check that they all have normal subgroups of order 2. Thus, just as above, we may assume that the degree of the covering map is 2. Then the exact sequence (\ref{E:ExactSequence}) takes the form

\[
1 \to O^* \to G \to \Z_2 \to 1.
\]

Now, $G$ must be a space form group of order 96, so by examining the list in Proposition~\ref{P:SpaceForms} we have the following possibilities: $\Z_{96}$, $D_{96}^*$, $\Z_3 \times D_{32}^*$, or $D'_{2^{5} \cdot 3}$. The first one can be ruled out because $O^*$ is not abelian, so it cannot be contained in $\Z_{96}$. We know from Proposition \ref{P:normal} that the subgroups of index 2 in $D_{96}^*$ are isomorphic to $D_{48}^*$, which is not isomorphic to $O^*$, so this case is also impossible. If $G$ were isomorphic to $\Z_3 \times D_{32}^*$, then Proposition~\ref{P:split} would imply that $O^*$ was contained in $D_{32}^*$, which is impossible because $|O^*| > |D_{32}^*|$. Finally, the Sylow 2-subgroups of $O^*$ are isomorphic to $D_{32}^*$ (by Remark~\ref{R:SpaceForms} one needs only observe that they are not cyclic), so if $O^*$ were contained in $D'_{2^5 \cdot 3}$, then $D_{32}^*$ would be also. When we look at the binary dihedral case later, we will see that this is not so. Thus, this case is also impossible.

\section{Binary tetrahedral}

We now suppose that $\pi_1^\infty(\tX)$ is the binary tetrahedral group $T^*$. Since $\chi(\tX) = 7$ (Table~\ref{Table:Kronheimer}), the only possibility for $\pi_1(X)$ is $\Z_7$, and the exact sequence becomes
\[
1 \to T^* \to G \to \Z_7 \to 1.
\]

Since the center and outer automorphism groups of $T^*$ each contain two elements, for the same reason as in the case of $I^*$, any extension of $T^*$ by a group of odd order splits as a direct product, so $G = \Z_7 \times T^*$. This actually \emph{is} a space form group, so we must rule out this case in another way.

It is clear that up to diffeomorphism we can represent $X$ as the interior of a compact 4-manifold with boundary $Y = S^3 / G$. Now, this $Y$ is a Seifert-fibered 3-manifold with Seifert invariants
\[
(b; (a_1, b_1), (a_2, b_2), (a_3, b_3)) =(0; (3, 4), (3, 4), (-2, 3)).
\]
(These invariants are well-defined only up to an equivalence relation, and we've made a choice that is suited to our purpose). Since
\[
   a_1 a_2 a_3 \left(\frac{b_1}{a_1} + \frac{b_2}{a_2} + \frac{b_3}{a_3} - b \right) = -21
\]
is odd, the space $Y$ is a $\Z_2$-homology sphere \cite{Saveliev}. Thus, $Y$ has a unique spin structure, and Rokhlin's theorem \cite{Rokhlin} implies that the quantity $\mu = \frac{\tau(M)}{8} \pmod{2}$, called the Rohklin invariant of $Y$,  is the same for any choice of spin 4-manifold $M$ with $\partial M = Y$. We will compute $\mu(Y)$ and show that $Y$ cannot possibly bound $X$.

We use the following formula \cite{Neumann}:

\begin{theorem}
Given a Seifert-fibered $\Z_2$-homology 3-sphere $N$ with Seifert invariants $(0; (a_1, b_1), (a_2, b_2), (a_3, b_3))$ such that exactly one $a_i$ is even and each $a_i - b_i$ is odd, the Rokhlin invariant of $N$ is given by
\[
   \mu(N) \equiv \left(\sum c(a_i - b_i, a_i) + \sign e(N)\right) / 8 \pmod{2},
\]
where:
\begin{itemize}
\item{$e(N) = \sum (b_i / a_i)$; and}
\item{$c$ is defined recursively by $c(a, \pm 1) = 0$ for odd $a$, $c(a \pm 2b, b) = c(a, b)$, $c(a, b + a) = c(a, b) + \sign b(b + a)$, and $c(a, b) = -c(-a, b) = -c(a, -b)$.}
\end{itemize}
\end{theorem}

We then compute:
\begin{align*}
\mu(Y) &\equiv \left[ c(-1, 3) + c(-1, 3) + c(-5, -2) + \sign\left(\frac{4}{3} + \frac{4}{3} - \frac{3}{2}\right) \right] / 8 \pmod{2}\\
       &\equiv -\frac{1}{4} \pmod{2}.
\end{align*}

Now, $X$ is spin, because $p^* w_2(X) = w_2(\tX) = 0$,  and $p^*$ is injective on $H^2(X, \Z_2)$ because the degree of the covering is odd. Thus, we must have
\begin{equation*}
\mu(Y) \equiv \frac{\tau(X)}{8} \pmod{2}.
\end{equation*}

But since $\chi(\tX) = 7$ and the degree of the covering map is also $7$, we have $\chi(X) = 1$ and so by Remark~\ref{R:Cohomology}, the signature of $X$ must be 0, which is a contradiction. Thus, this case is ruled out.

\section{Binary dihedral case}

Suppose now that $\pi_1^\infty(\tX)$ is the binary dihedral group $D_{4b}^*$. Just as before, we have an exact sequence
\[
1 \to D_{4b}^* \to G \to H \to 1,
\]
where $G$ is a space form group.

\subsection{Possibilities for $G$}

We first determine which space form groups $G$ contain $D_{4b}^*$ as a normal subgroup.  By Proposition~\ref{P:split}, we consider first only those $G$ which do not split off a cyclic factor. We saw in Proposition \ref{P:normal} that the groups $T^*$, $O^*$, and $I^*$ do not contain normal subgroups of the form $D_{4b}^*$ for $b > 2$, but they do contain subgroups isomorphic to $D_8^*$. The same proposition tells us that the only possibilities with $G$ itself a binary dihedral group are $D_{4b}^*$ or $D_{4\cdot2b}^*$. We thus only have two cases left to consider: $D'_{2^{k + 2} p}$ and $T^*_{8 \cdot 3^v}$. We use the presentation
\begin{equation} \label{E:BinaryDihedral}
D_{4b}^* = \langle\, z, w \mid z^{2b} = 1, z^b = w^2, wzw^{-1} = z^{-1} \,\rangle.
\end{equation}

\begin{prop}
The groups $D'_{2^{k + 2} p}$ contain no binary dihedral subgroups (see Remark~\ref{R:Conventions}).
\end{prop}
\begin{proof}
We use the presentation
\begin{equation}\label{E:D'}
D'_{2^{k + 2} p} = \langle\, x, y \mid x^{2^{k + 2}} = y^p = 1, yx = xy^{-1} \,\rangle.
\end{equation}

Notice that every element of this group can be written uniquely as $x^c y^d$ for $c \in \Z_{2^{k + 2}}$ and $d \in \Z_p$, and that the unique element of order 2 in this group is $x^{2^{k + 1}}$. Thus, the elements of order 4 in this group are exactly those which square to $x^{2^{k + 1}}$. Suppose that $x^c y^d$ is such an element. Then
\[
(x^c y^d)^2 = x^{2c} y^{d \cdot (-1)^c + d} = x^{2^{k + 1}}
\]
implies that $2c \equiv 2^{k + 1} \pmod{2^{k + 2}}$, so in particular $c$ is even (since by convention $k > 0$). Thus $2d \equiv 0 \pmod{p}$, and since $p$ is odd, this means that $d = 0$. Thus, every element of order 4 is an even power of $x$. But the relations in (\ref{E:D'}) imply that all even powers of $x$ are in the center of the group.

On the other hand, every binary dihedral group has at least one element of order 4, given by $w$ in the presentation (\ref{E:BinaryDihedral}), but the center of the group contains only one non-identity element, which has order 2. The result follows.
\end{proof}

\begin{prop}
The groups $T'_{8 \cdot 3^v}$ contains no binary dihedral subgroups $D_{4b}^*$ for $b > 2$.
\end{prop}
\begin{proof}
We use the presentation
\begin{align} \label{E:Tv}
T'_{8 \cdot 3^v} = \langle\, x, p, q \mid & x^{3^v} = p^4 = 1, p^2 = q^2, xqx^{-1} = pq, \nonumber \\
~ & xpx^{-1} = q, pqp^{-1} = q^{-1} \,\rangle.
\end{align}

Suppose that we had an injective homomorphism $\phi: D_{4b}^* \to T'_{8 \cdot 3^v}$. Then we must have $4b \mid 8 \cdot 3^v$, so $b = 3^k$ or $b = 2 \cdot 3^k$ for some $k > 0$. From the presentation (\ref{E:BinaryDihedral}) we see that $z^p$ has order $3^k$, where $p = 2$ if $b$ is odd and $p = 4$ if $b$ is even. In either case, since $p$ divides $2b$ but not $b$, the subgroup $\langle z^p \rangle$ does not contain $z^b$, and so $\langle z^p \rangle$ intersects the center of the group only at the identity.

Returning to the presentation (\ref{E:Tv}), we see that $\langle x \rangle$ is a Sylow 3-subgroup, and a quick calculation shows that $x^3$ commutes with $p$ and $q$, so $\langle x^3 \rangle$ is contained in the center of the group. Thus, by modifying $\phi$ by an inner automorphism of $T'_{8 \cdot 3^v}$, we may assume that the image of $\langle z^p \rangle$ lies in $\langle x \rangle$. However, $\langle x^3 \rangle$ is the unique maximal subgroup of $\langle x \rangle$, so the image of $\langle z^p \rangle$ must intersect it non-trivially, which gives a contradiction.
\end{proof}

Thus, for $b > 2$, the only possibilities for $G$ which do not split off a cyclic factor are $D_{4b}^*$ and $D_{4\cdot 2b}^*$. For $b = 2$, we had the additional possibilities $T^*$, $O^*$, $I^*$, and $T'_{8 \cdot 3^v}$. However, in this case we see from Table~\ref{Table:Kronheimer} that $\chi(\tX) = 5$, so we can restrict our attention to those space form groups which contain $D_8^*$ as a normal subgroup of index 5, which rules out all of these extra possibilities. Therefore, for any $b$, either $G = \Z_m \times D_{4b}^*$ or $G = \Z_m \times D_{4 \cdot 2b}^*$. In the latter case, since $m$ is required to be odd and the abelianization of $D_{4b}^*$ has even order (either $\Z_2 \times \Z_2$ or $\Z_4$, depending on the parity of $b$), the composition of the injection $D_{4b}^* \hookrightarrow \Z_m \times D_{4\cdot 2b}^*$ with the projection onto the first factor is the zero map. Thus in this case, $\pi_1(X) = H = \Z_m \times \Z_2$, so by replacing $X$ with its covering space determined by the subgroup $\Z_m$, we may assume that $G = D_{4\cdot2b}^*$ and $H = \Z_2$.

\subsection{Eta invariants}

Since we were unable to eliminate every case when $\pi_1^\infty(\tX)$ is binary dihedral by looking at the possible space form groups, we must rule out the remaining cases by another method.

Recall the signature formula for Riemannian manifolds with boundary (see \cite{Atiyah}):
\begin{equation}\label{E:Signature}
\tau(M) = \frac{1}{12\pi^2} \int_X |W_+|^2 - |W_-|^2 d\mu - \eta(\partial M).
\end{equation}
This applies to our situation because this formula is conformally invariant, and our ALE manifolds can be made into compact manifolds with boundary by a conformal rescaling. The quantity $\eta(\partial M)$ is an invariant of $\partial M$ (for appropriate boundary conditions on the metric of $M$), built up from the spectrum of a certain differential operator. We will compute it by a different method.

Namely, given a spherical space form $S^3 / G$ for some finite subgroup $G < SO(4)$,  we have \cite[I.6]{Donnelly}
\begin{equation*}
\eta(S^3) - |G| \eta(S^3 / G) = - \sum_{g \neq 1} \eta_g(S^3),
\end{equation*}
where the sum is taken over the non-identity elements in $G$. If we think of $S^3$ as the boundary of the unit ball in $\R^4$ and consider the action of $G$ on the ball, then since the action has exactly one fixed point at the origin and the ball has no second cohomology, the quantity $\eta_g(S^3)$ is given by \cite[I.4]{Donnelly}
\begin{equation*}
\eta_g(S^3) = -\cot(\theta_1 / 2) \cot(\theta_2 / 2),
\end{equation*}
where for each $g$ we split the tangent space at the origin into two 2-planes, on which $dg$ acts by rotation through an angle $\theta_1$ and $\theta_2$, respectively. If $dg$ is actually complex-linear, then this just means that the complex eigenvalues are $\exp(\theta_1)$ and $\exp(\theta_2)$, where $\exp(x) := e^{2 \pi i x}$.

Since $S^3$ admits an orientation-reversing isometry, $\eta(S^3) = 0$, and we get the formula
\begin{equation*}
\eta(S^3 / G) = -\frac{1}{|G|} \sum_j \cot(\theta_{1j} / 2) \cot(\theta_{2j} / 2).
\end{equation*}

For an element $A \in U(2)$ (thought of as a matrix), we will denote by $\eta(A)$ the quantity $-\cot(\theta_1 / 2) \cot(\theta_2 / 2)$ described above. We will work only with matrices that are either diagonal or anti-diagonal. The second case turns out to be trivial:

\begin{claim} \label{C:anti-diagonal}
For any $A \in U(2)$ of the form
\[
\begin{bmatrix}
0 & \exp(a)\\
\exp(b) & 0\\
\end{bmatrix},
\]
we have $\eta(A) = 1$.
\end{claim}
\begin{proof}
This matrix has eigenvalues $\exp(\frac{a + b}{2})$ and $\exp(\frac{a + b}{2} + \frac{1}{2})$. Thus
\[
   \eta(A) = -\cot\left(\frac{\pi(a + b)}{2}\right) \cot\left(\frac{\pi(a + b)}{2} + \frac{\pi}{2}\right).
\]
But the angle-addition formula for cotangent gives $\cot(x + \frac{\pi}{2}) = -1 / \cot(x)$, and the claim follows.
\end{proof}

\subsection{Calculation of eta invariants}

Consider first a 3-dimensional spherical space form $Y$ with fundamental group $\Z_u \times D_{4v}^*$ for some even $v$. According to Wolf \cite{Wolf}, up to equivalence and automorphisms of the group, the only representation of this group into $SO(4)$ which acts freely on the 3-sphere has image in $U(2) < SO(4)$ generated by the matrices
\[
A =
\begin{bmatrix}
\exp(\frac{2v + u}{2uv}) & 0\\
0 & \exp(\frac{2v - u}{2uv})\\
\end{bmatrix}
\]
and
\[
B =
\begin{bmatrix}
0 & 1\\
-1 & 0\\
\end{bmatrix}.
\]
This representation comes from the presentation
\[
   \langle\, A, B \mid A^{2uv} = 1, B^2 = A^{uv}, BAB^{-1} = A^k \,\rangle,
\]
where $k$ is chosen so that $k \equiv -1 \pmod{2v}$ and $k \equiv 1 \pmod{u}$. (This choice is unique by the Chinese remainder theorem).

We wish to compute the quantity $\eta(Y) = \frac{1}{4uv} \sum_{g \in G} \eta(g)$, identifying $G$ with the image of the representation above. Notice that every element of $G$ can be uniquely written as $A^p B^q$ for some $p \in \Z_{2uv}$ and $q \in \Z_2$. If $q$ is 0, then the element is diagonal, and if $q$ is 1, then the element is anti-diagonal. Because of Claim~\ref{C:anti-diagonal}, each anti-diagonal element contributes $1$ to the above sum, and there are $2uv$ of them, so $\eta(Y) = \frac{1}{2} + \frac{1}{4uv} \sum_{i = 0}^{2uv - 1} \eta(A^i)$. But since each $A^i$ is diagonal, this latter sum is equal to
\[
-\frac{1}{4uv} \sum_{i = 1}^{2uv - 1} \cot\left(\frac{\pi i (2v + u)}{2uv}\right) \cot\left(\frac{\pi i (2v - u)}{2uv}\right),
\]
which, as noticed by Hirzebruch \cite{Hirzebruch} has a nice interpretation in terms of so-called Dedekind sums. These are certain finite sums which were first investigated by Dedekind in relation to transformation properties of modular forms. They are defined as follows:

Let $\ded{.}: \R \to \R$ be the ``sawtooth'' function given by:
\[
\ded{x} =
\begin{cases}
x - \lfloor x \rfloor - \frac{1}{2} & x \in \R \setminus \Z\\
0 & x \in \Z
\end{cases},
\]
where $\lfloor x \rfloor$ means the floor of $x$ (the greatest integer $ \leq x$). Then the Dedekind sum $D(a, b; c)$ for pairwise relatively-prime integers $a$, $b$, and $c$ is
\begin{equation} \label{E:dedekind}
D(a, b; c) = \sum_{i = 1}^{c - 1} \bigded{\frac{ai}{c}} \bigded{\frac{bi}{c}}.
\end{equation}

The sums studied by Dedekind were of the form $s(b, c) := D(1, b; c)$. However, whenever $a$ and $c$ are relatively prime integers, we can find an inverse $z$ for $a$ modulo $c$, and we see by the periodicity of the sawtooth function that $D(a, b; c) = D(1, bz; c) = s(bz, c)$, so these 2-argument sums are no less general.

The relationship between these sums and our cotangent sum (see for example \cite{Rademacher}) is that
\[
   D(p, q; r) = \frac{1}{4r} \sum_{k = 1}^{r - 1} \cot\left(\frac{\pi pk}{r}\right) \cot\left(\frac{\pi qk}{r}\right).
\]

Thus the $\eta$ invariant of $Y$ is just given by
\[
   \eta(Y) = \frac{1}{2} - 2 D(2v + u, 2v - u, 2uv).
\]

We will compute this Dedekind sum shortly, but let us first examine the case where the fundamental group of $Y$ is $\Z_u \times D_{4v}^*$ for odd $v$. There is again just one orthogonal representation up to equivalence and automorphisms, given in a complex basis by the generators
\[
A = 
\begin{bmatrix}
\exp(1 / v) & 0\\
0 & \exp(-1 / v)\\
\end{bmatrix}
\]
and
\[
B = 
\begin{bmatrix}
0 & 1\\
\exp(1 / 2u) & 0\\
\end{bmatrix}.
\]

Every element of the group can be written uniquely as $A^sB^t$ with $s \in \Z_v$ and $t \in \Z_{4u}$. It is easy to see that $B^2$ commutes with $A$, and that the set of elements with $t$ even is a cyclic subgroup generated by
\[
AB^2 =
\begin{bmatrix}
\exp\left(\frac{v + 2u}{2uv}\right) & 0\\
0 & \exp\left(\frac{v - 2u}{2uv}\right)\\
\end{bmatrix}.
\]
The elements with $t$ odd are all anti-diagonal, and by the same reasoning as in the even case, we obtain the formula
\[
\eta(Y) = \frac{1}{2} - 2D(v + 2u, v - 2u, 2uv) = \frac{1}{2} + 2D(2u + v, 2u - v, 2uv).
\]

Thus, if we can compute $D(2x + y, 2x - y, 2xy)$ for $2x$ and $y$ relatively prime, then we will have a formula for both cases. In general, there is no closed-form expression for Dedekind sums, but in this instance, it will turn out that we can get an almost-explicit answer. We begin with the Rademacher reciprocity formula \cite{Rademacher2}:
\[
D(a, b; c) + D(b, c; a) + D(c, a; b) = \frac{1}{12} \frac{a^2 + b^2 + c^2}{abc} - \frac{1}{4}
\]
This reduces our problem to that of computing $D(2x - y, 2xy; 2x + y)$ and $D(2xy, 2x + y; 2x - y)$. We will sketch the calculation of the first; the second is similar.

Clearly, $D(a, b; c)$ depends only on the residue class of $a$ and $b$ modulo $c$. Moreover, since we sum over all residue classes modulo $c$, if $d$ is relatively prime to $c$, then $D(ad, bd; c) = D(a, b; c)$. Thus
\begin{align*}
D(2x - y, 2xy; 2x + y) &= D(-2y, 2xy; 2x + y)\\
~ &= D(-1, x; 2x + y)\\
~ &= -D(1, x; 2x + y)\\
~ &= -s(x, 2x + y),
\end{align*}
where we used the obvious identity $D(-a, b; c) = -D(a, b; c)$. To evaluate this last 2-argument sum, we use the following special case of Rademacher reciprocity, which was actually first proved by Dedekind \cite[vol. 1, pp. 159-173]{Dedekind}:
\[
s(b, c) + s(c, b) = \frac{1}{12}\left(\frac{b}{c} + \frac{1}{bc} + \frac{c}{b}\right) - \frac{1}{4}.
\]
This formula allows us to compute $s(2x + y, x)$ instead, but this is just the same as $s(y, x)$. Performing the same steps for the other term $D(2xy, 2x + y; 2x - y)$ and putting everything together, we have, after the dust settles, the following:
\[
D(2x + y, 2x - y, 2xy) = \frac{1}{12xy} + \frac{y}{6x} - \frac{1}{4} - 2s(y, x).
\]

Applying this in the odd case, we get:
\[
\eta(Y) = \frac{1}{2} + 2D(2u + v, 2u - v, 2uv) = \frac{1}{6uv} + \frac{v}{3u} - 4s(v, u),
\]
and in the even case:
\[
\eta(Y) = \frac{1}{2} - 2 D(2v + u, 2v - u, 2vu) = 1 - \frac{1}{6vu} - \frac{u}{3v} + 4s(u, v).
\]

But by applying Dedekind reciprocity to the Dedekind sum in this last expression, we get exactly the same formula as in the odd case! Thus, we have:
\begin{lemma}
If $Y$ is a 3-dimensional spherical space form with fundamental group $\Z_u \times D_{4v}^*$, then
\begin{equation} \label{E:Eta0}
\eta(Y) = \frac{1}{6uv} + \frac{v}{3u} - 4 s(v, u).
\end{equation}
\end{lemma}

Now we return to our original context. We were interested in two cases: either $G = D_{4 \cdot 2b}^*$, or $G = \Z_m \times D_{4b}^*$. In the former case, $u = 1$, and the Dedekind sum in the above formula vanishes. In the latter case, we have a space $\tX$ with Euler characteristic $b + 3$ (Table~\ref{Table:Kronheimer}), and we have the group $\Z_m$ acting as a group of covering transformations, so we must have $m \mid b + 3$. This implies that $b \equiv -3 \pmod{m}$, so that the Dedekind sum $s(b, m)$ is equal to $s(-3, m) = -s(3, m)$. Finally, the Dedekind reciprocity formula allows us to write this last sum in terms of either $s(0, 3)$, $s(1, 3)$, or $s(2, 3)$, depending on the residue class of $m$ modulo 3, but these three sums are easy to evaluate by hand. Thus, we determine that
\begin{equation} \label{E:Eta}
\eta(\partial X) = \frac{1}{6mb} + \frac{b}{3m} +
\begin{cases}
\frac{(m - 10)(m + 1)}{9m} & m \equiv 0 \pmod{3}\\
\frac{(m - 10)(m - 1)}{9m} & m \equiv 1 \pmod{3}\\
\frac{(m - 5)(m - 2)}{9m} & m \equiv 2 \pmod{3}\\
\end{cases}.
\end{equation}

Now, in the anti-self-dual case the signature formula (\ref{E:Signature}) reduces to
\[
\tau(X) = -\frac{1}{12\pi^2} \int_X |W_-|^2 \, d\mu - \eta(S^3 / G).
\]

We can compute the first two terms of this equation directly. Let $d$ be the degree of the covering map. We know that $\chi(\tX) = b + 3$, so we must have $\chi(X) = \frac{b + 3}{d}$. By Remark~\ref{R:Cohomology}, this means that $\tau(X) = 1 - \frac{b + 3}{d}$. The Gauss-Bonnet formula for bubbles reduces to
\begin{equation} \label{E:GaussBonnet}
\chi(M) = \frac{1}{8\pi^2} \int |W_-|^2 d\mu + \frac{1}{|\pi_1^\infty|},
\end{equation}
which implies that the integral of $|W_-|^2$ over $\tX$ is equal to $8\pi^2(b + 3 - \frac{1}{4b})$. Thus the integral of the same quantity over $X$ is $\frac{1}{d}$ times this. Putting this all together, we get that the eta invariant of the boundary of $X$ must be given by
\begin{equation} \label{E:Eta2}
\eta(\partial X) = \frac{b}{3d} + \frac{1}{d} + \frac{1}{6bd} - 1.
\end{equation}

If the fundamental group at infinity of $X$ is $\Z_m \times  D_{4b}^*$, then setting Equation~\ref{E:Eta} and Equation~\ref{E:Eta2} with $d = m$ equal, we must have
\[
\frac{1}{m} - 1 =
\begin{cases}
\frac{(m - 10)(m + 1)}{9m} & m \equiv 0 \pmod{3}\\
\frac{(m - 10)(m - 1)}{9m} & m \equiv 1 \pmod{3}\\
\frac{(m - 5)(m - 2)}{9m} & m \equiv 2 \pmod{3}\\
\end{cases}.
\]
It is easily verified that the only positive integer solution of these equations is $m = 1$, so there are no non-trivial quotients in this case. Similarly, in the case with $\pi_1^\infty(X) = D_{4 \cdot 2b}^*$, we must have, setting equal Equation~\ref{E:Eta0} with $v = 2b$ and $u = 1$  (so that the Dedekind sum is 0), and Equation~\ref{E:Eta2} with $d = 2$:
\[
\frac{1}{12b} + \frac{2b}{3} = \frac{1}{12b} + \frac{b}{6},
\]
which clearly has no positive-integer solutions. Thus, we finally conclude that there are no non-trivial quotients in the binary dihedral case.

\section{Cyclic case}

We will work from the other direction in this case: starting with the universal covering and determining all isometric quotients.

\subsection{Gibbons-Hawking ansatz}

Let $(M, g, \pi)$ be a Gibbons-Hawking space with $k$ monopoles, i.e., given by the following ansatz:
\begin{theorem}[Gibbons-Hawking \cite{Gibbons-Hawking}]
Let $\pi: M_0 \to \R^3 - \{p_1, \ldots, p_k\}$ be the principal $S^1$-bundle whose first Chern class yields $-1$ when paired with the generators of second homology given by small 2-spheres about the ``monopole'' points $\{p_1, \ldots, p_k\}$. Let $V$ be the function on $\R^3$ given by
\[
	V(x) = \frac{1}{2} \sum_{i = 1}^k \frac{1}{|x - p_i|}.
\]
If we equip the bundle $M_0$ with the connection 1-form $\omega$ defined by $d\omega = \pi^*(*dV)$, then the metric
\[
g = V(dx_1^2 + dx_2^2 + dx_3^2) + V^{-1} \omega^2
\]
on $M_0$ can be smoothly completed, yielding a hyperk\"ahler manifold $M$.
\end{theorem}

The difference $M \setminus M_0$ consists of $k$ points $\{\tilde{p}_1, \ldots, \tilde{p}_k\}$, and the map $\pi$ on $M_0$ extends to a map $\pi: M \to \R^3$ such that $\pi(\tilde{p}_i) = p_i$. To see the $S^2$ of parallel complex structures, first let $\partial_\theta$ denote the vertical vector field on $M_0$ defined by $\omega(\partial_\theta) = 1$. Then the vector fields $V^{1/2}\partial_\theta$ and $V^{-1/2}\partial_{x_i}$ (actually the horizontal lifts) give an orthonormal trivialization of the tangent bundle of $M_0$. Thus an orthogonal almost-complex structure $I$ on $M_0$ which is compatible with the orientation is uniquely determined by where it sends $V^{1/2}\partial_\theta$, which must be a unit vector in $\Span\{V^{-1/2}\partial_{x_i}\} \cong \R^3$. One needs only check that these almost-complex structure all extend to $M$, and that the result is integrable.

Notice also that the $S^1$ action coming from the principal bundle structure on $M_0$ extends to an action on $M$ by leaving $M \setminus M_0$ fixed pointwise, so that $\partial_\theta$ extends by zero to $M$. One can check from the above description that this $S^1$ action is actually triholomorphic. It also follows from this description that $\pi$ is both the (ordinary) quotient map and the hyperk\"ahler moment map for this $S^1$ action (see \cite{Hitchin}).

We will determine the possible isometries of $M$ by looking at holomorphic curves in $M$:

\begin{prop} \label{P:Curves}
Every real surface $C \subset M$ which is a $(-2)$-curve with respect to some parallel complex structure is the preimage under $\pi$ of a straight line segment between monopoles.
\end{prop}
\begin{proof}
Equip $M$ with the parallel complex structure $I$ that makes $C$ into a $(-2)$-curve. Since $\partial_\theta$ is triholomorphic, the translations of $C$ along that field are also $(-2)$-curves, and they are homologous to the original, so they must actually coincide. Thus, $C$ is the union of fibers of $\pi$. For $p \in \R^3$ such that $C$ contains $\pi^{-1}(p)$, we see that $I(\partial_\theta)_x$ is tangent to $C$ for each $x \in \pi^{-1}(p)$. But by construction of $M$, $\pi_*(I(\partial_\theta)_x)$ is independent of $x \in \pi^{-1}(p)$. Let $l$ be the line in $\R^3$ containing $p$ and tangent to $\pi_*(I(\partial_\theta)_x)$. Then $\pi^{-1}(l)$ is a non-compact holomorphic curve with respect to $I$, and its intersection with $C$ contains a circle, so it must be all of $C$. Thus $C \subset \pi^{-1}(l)$. If $\{p_{i_1}, \ldots, p_{i_m}\}$ are the monopoles lying on $l$, then $\pi^{-1}(l)$ is a chain of $m - 1$ spheres intersecting transversely, together with two discs intersecting the end spheres transversely. It follows that $C$ must be the preimage of a segment between monopoles.
\end{proof}

\begin{corollary} \label{C:Triholomorphic}
If $k > 2$, then the identity component of the group of triholomorphic isometries of $M$ is isomorphic to $U(1)$.
\end{corollary}
\begin{proof}
Take 3 distinct monopoles $p, q, r$ and assume that the line segments $pq$ and $qr$ intersect only at $q$ (by reordering). The preimage of these two segments will be two spheres intersecting transversely.

Now, any isometry of $M$ takes a parallel complex structure to another one, so it must permute the union of all real surfaces which are $(-2)$-curves for some parallel complex structure. Thus the identity component of the triholomorphic isometry group acts by homeomorphisms isotopic to the identity on the union of these curves. In particular, every element must fix the point $\pi^{-1}(q)$. Since isometries are determined by their value and differential at any single point, we have a faithful representation of the identity component of the triholomorphic isometry group into $GL(T_{\pi^{-1}q}M)$. Identify this tangent space with $\C^2$ in a way that is compatible with the parallel complex structure making the preimage of the segment between $p$ and $q$ into a $(-2)$-curve. Then the image of the representation must lie in $SU(2)$. Moreover, it must take the tangent spaces of the two intersecting spheres into themselves. Thus the image must actually lie in $\{\, \diag(a, a^{-1}) \mid a \in U(1)\,\} \cong U(1)$.
\end{proof}

\subsection{Isometries}

Let $G$ be a finite group of isometries of $M$ acting freely, and let $f$ be an element of $G$. Since $\tau(M) \neq 0$, we may assume that every element of $G$ acts in an orientation-preserving way.

We first consider the case $k > 2$. First of all, by Corollary~\ref{C:Triholomorphic}, the pushforward of $\partial_\theta$ must be a multiple of itself: $f_* \partial_\theta = a \partial_\theta$ for some $a \in \R^\times$. But since $f$ has finite order, $a = \pm 1$. Notice that since $\partial_\theta$ is exactly tangent to the fibers of $\pi$, $f$ must take fibers of $\pi$ to fibers, and so we get an induced map $\hat{f}: \R^3 \to \R^3$ on the quotient space. If $a = 1$, then $f$ preserves the orientation of the circle fibers, and if $a = -1$, then it reverses them.

If instead $k = 2$ (which is the Eguchi-Hanson metric \cite{Eguchi-Hanson}), then the triholomorphic isometry group is $SO(3)$ (see \cite{Gibbons-Hawking2}), so this same argument doesn't work. From Proposition~\ref{P:Curves}, it follows that $G$ must preserve the 2-sphere $S$ given by the preimage of the segment between the two monopoles and act freely on it, so $G$ must be isomorphic to $\Z_2$. Let $A \in GL(3; \R)$ be the action of the non-identity element on the space of triholomorphic Killing fields $\mathfrak{so}(3) \cong \R^3$. Then $A$ must have a nonzero eigenvector, but $A^2 = \Id$ implies that the only possible eigenvalues are $\pm 1$. This implies that $G$ must send \emph{some} triholomorphic Killing field to itself or its negative. However, consider the explicit form
\begin{equation*}
g = \left[1 - \left(\frac{a}{4}\right)^4\right]dr^2 + r^2\left[1 - \left(\frac{a}{4}\right)^4\right](d\psi + \cos\theta d\phi)^2 + \frac{r^2}{4}(d\theta^2 + \sin^2\theta d\phi^2)
\end{equation*}
for the metric, where $\theta$, $\phi$, and $\psi$ are Euler angles on $SO(3)$. The triholomorphic $SO(3)$ in this picture is exactly the standard $SO(3)$ acting on the round 2-sphere $r = a$, which corresponds to $S$ in the previous picture. It then follows that the triholomorphic isometry group acts transitively on the unit sphere of the triholomorphic Killing fields, so by conjugating $G$ by some triholomorphic isometry, we may assume that $G$ sends $\partial_\theta$ to itself or its negative. Thus, we are in exactly the same situation as before: $f_* \partial_\theta = \pm \partial_\theta$, and there is an induced map $\hat{f}: \R^3 \to \R^3$.

Since $\pi$ is also the hyperk\"ahler moment map for the $S^1$ action, there are parallel complex structures $I$, $J$, and $K$ with K\"ahler forms $\omega_I$, $\omega_J$, and $\omega_K$ such that the 1-forms defined by $dx = \omega_I(\partial_\theta, \cdot)$, $dy = \omega_J(\partial_\theta, \cdot)$, $dy = \omega_K(\partial_\theta, \cdot)$ are actually the differentials of the component functions of $\pi: M \to \R^3$.

Now, $f$ acts on the space of parallel self-dual 2-forms by pullback, so in the basis for this space given by $\omega_I$, $\omega_J$, and $\omega_K$, we can think of this action as being an element $\rho(f)$ of $SO(3)$. Let $d\vec{x}$ be the $\R^3$-valued 1-form on $M$ with components $(dx, dy, dz)$, and $\vec{\omega}$ the $\R^3$-valued 2-form with components $(\omega_I, \omega_J, \omega_K)$. We compute:
\begin{eqnarray*}
(f^* d\vec{x})(v) &=& f^*(\vec{\omega}(\partial_\theta, \cdot))(v)\\
~&=& \vec{\omega}(\partial_\theta, f_* v)\\
~&=&\pm (f^* \vec{\omega})(\partial_\theta, v)\\
~&=& \pm (\rho(f) \cdot \vec{\omega}) (\partial_\theta, v)\\
~&=& \pm \rho(f) \cdot d\vec{x} (v),
\end{eqnarray*}
where the sign in the same as in $f_* \partial_\theta = \pm \partial_\theta$.

Thus, if we go back to the induced map $\hat{f}: \R^3 \to \R^3$, we see that it has constant differential, given by $\pm \rho(f)$. Thus $\hat{f}$ is an affine transformation. However, since $f$ has finite order and $d\hat{f} \in O(3)$, this implies that $\hat{f}$ is conjugate in the Euclidean group to an element of $O(3)$: namely, it is given by a rotation around some point $o \in \R^3$.

Now if any of the fixed points of $\hat{f}$ were were monopoles, then since the preimage of each monopole under $\pi$ is only one point, $f$ would have a fixed point on $M$, so this is impossible. Thus $o$ is not a monopole point. Consider the action of $f$ on the circle fiber over that point. If $f_* \partial_\theta = - \partial_\theta$, then $f$ reverses the orientation of the circle. But every orientation-reversing homeomorphism of $S^1$ has a fixed point, so this is also impossible. Therefore, we can assume that $\partial_\theta$ is invariant under $f$.

Repeating the steps above for each element of $G$, we get an action of $G$ on $\R^3$ by affine transformations, and since $G$ is finite, the image of $G$ in the Euclidean group must be conjugate to a subgroup of $SO(3)$ (take the center of mass of some orbit to find an appropriate origin). By modifying $\pi$ by a translation, we will assume that the image of $G$ actually \emph{is} a subgroup of $SO(3)$. Notice that since each element of $G$ takes fibers of $\pi$ to fibers, and monopoles have 1-point preimages, the image of $G$ in $SO(3)$ must permute the monopoles. Moreover, since $G$ acts freely on $M$, the image in $SO(3)$ must act freely on the monopoles. This proves one direction of Theorem~A.

Now suppose that we have a finite subgroup $G$ of $SO(3)$ which acts freely on the monopole points (so we're taking the origin to be the center of mass of the monopoles). The lift of any single $g \in G$ to an isometry of $M$ follows exactly as in \cite[Proposition 2.7]{Honda}: such a lift always exists, and is unique up to the global $S^1$ action. In particular, we have an exact sequence
\begin{equation*}
1 \to S^1 \to \Aut(M, G) \to G \to 1,
\end{equation*}
where $\Aut(M, G)$ denotes the isometries of $M$ which descend to an element of $G$ on $\R^3$. But $G$ fixes the origin, which is a non-monopole point. In particular, each element of $\Aut(M, G)$ acts isometrically on the circle which lies over the origin, giving a splitting homomorphism $\Aut(M, G) \to S^1$. Thus $\Aut(M, G) = S^1 \times G$.

Thus, a choice of lifting for $G$ is equivalent to a homomorphism $G \to S^1$. If we want the lifting to act freely, then the map should be injective, since the image of $g \in G$ in $S^1$ represents the action of $g$ on the circle fiber over the origin. Thus $G$ must be cyclic, and the homomorphism is unique up to an automorphism of $G$.

Every cyclic subgroup of $SO(3)$ is a group of rotations about some fixed axis, so the action of $G$ will preserve some constant vector field $w$ on $\R^3$. Since it permutes the monopoles, it will also preserve the function $V$, so $V^{-1/2} w$ will be invariant. We already saw that $G$ preserves $\partial_\theta$ and the orientation, so the action of $G$ must actually preserve the complex structure given by $I(\partial_\theta) = V^{-1/2} w$, and thus the quotient $M / G$ is always K\"ahler.

This concludes the proof of Theorem~A.

\section{Consequences} \label{S:Consequences}

We now use the classification theorem to prove the two corollaries mentioned earlier:

\begin{corollaryb}
If $X$ is a Ricci-flat anti-self-dual ALE 4-manifold which is not flat, then
\[
\int_X |\Rm|^2 d\mu \geq
\begin{cases}
6\pi^2 & \text{if } b_2(X) = 0\\
8\pi^2 \left(b_2(X) + 1 - \frac{1}{b_2(X) + 1} \right) & \text{if } b_2(X) > 0.
\end{cases}
\]
In particular, $\int_X |\Rm|^2 d\mu \geq 6\pi^2$ in all cases.
\end{corollaryb}
\begin{proof}
The only component of the curvature tensor of $X$ which may not vanish is the anti-self-dual Weyl curvature, so in this case the Gauss-Bonnet formula (\ref{E:GaussBonnet}) gives
\begin{equation*}
\int_X |\Rm|^2 d\mu = 8\pi^2\left(\chi(X) - \frac{1}{|\pi_1^\infty(X)|}\right),
\end{equation*}
which is equal to
\begin{equation*}
8\pi^2 \left(b_2(X) + 1 - \frac{1}{|\pi_1^\infty(X)|} \right)
\end{equation*}
by Remark~\ref{R:Cohomology}. Evidently, in order to prove the corollary, it is enough to find the minimum of $|\pi_1^\infty|$ for each fixed $b_2$. From Table~\ref{Table:Kronheimer}, it is clear that the Gibbons-Hawking spaces dominate the other simply-connected bubbles in this respect, for each $b_2 \geq 1$, so we need only compare the Gibbons-Hawking spaces and their quotients. From the same table, we see that a $d$-fold quotient of a Gibbons-Hawking space with $k - 1$ monopoles has $|\pi_1^\infty| = dk$ and $b_2 = \frac{k}{d} - 1$, so $|\pi_1^\infty| = d^2(b_2 + 1)$, which is increasing with $d$. If $b_2 > 0$, then for each $d \geq 1$, there is such a space (set $k = d(b_2 + 1)$), so $|\pi_1^\infty| \geq b_2 + 1$, and the formula follows for that case. However, if $b_2 = 0$, then $d = 1$ implies $k = d(b_2 + 1) = 1$ and the 1-monopole Gibbons-Hawking space is just flat $\R^4$. Thus, the smallest non-flat choice for $d$ is 2, so $|\pi_1^\infty| \geq 4$ and the result follows.
\end{proof}

\begin{corollaryc}
If $X$ is a scalar-flat K\"ahler ALE 4-manifold (of any order $\tau > 0$) and $b_2(X) = 0$, then $X$ is actually Ricci-flat, and is a quotient of a Gibbons-Hawking space $\tX$ by a cyclic group of order $\chi(\tX)$.
\end{corollaryc}
\begin{proof}
Consider the Ricci form $\rho$ of $X$. Since $X$ is scalar-flat, the Ricci form must be harmonic \cite[2.33]{Besse}, and since $X$ is ALE, it must moreover be square-integrable. However, on an ALE 4-manifold, the space of $L^2$ harmonic forms is isomorphic to the image of $H^2(X, \partial X; \R)$ in $H^2(X; \R)$ \cite[Theorem 1A]{Hausel}. Therefore, if $b_2(X) = 0$, then $\rho$ must be the zero form, and so $X$ is Ricci-flat. The last assertion then follows immediately from the main theorem and Remark~\ref{R:Cohomology}
\end{proof}

\bibliographystyle{plain}
\bibliography{GI}

\vspace{1in}

\noindent 
{\bf Acknowledgement.} The author would like to thank Claude LeBrun for introducing him to this problem and for his support.

\end{document}